\documentclass[a4paper]{article}
\usepackage{fullpage}
\usepackage{amsmath}
\usepackage{amssymb}
\usepackage{amsthm}
\usepackage{graphicx}
\usepackage{bbm}
\usepackage{enumerate}
\usepackage{microtype}
\usepackage{xcolor}
\usepackage[none]{hyphenat}
\usepackage[colorlinks=true,urlcolor=blue,linkcolor=blue,plainpages=false,pdfpagelabels]{hyperref}
\allowdisplaybreaks

\newtheorem{theorem}{Theorem}[section]

\newtheorem{lemma}[theorem]{Lemma}

\newcommand{\N}{\mathbb{N}}

\newcommand{\T}{\mathbb{T}}

\newcommand{\vp}{\varphi}
\newcommand{\uu}[1]{\underline{#1}}
\newcommand{\wt}[1]{\widetilde{#1}}

\newcommand{\ar}{\mathrm{ar}}
\newcommand{\tp}{\mathrm{tp}}
\newcommand{\Sh}{\mathrm{Sh}}

\title{On extracting variable Herbrand disjunctions}
\author{Andrei Sipo\c s${}^{a,b}$\\[2mm]
\footnotesize ${}^a$Research Center for Logic, Optimization and Security (LOS), Department of Computer Science,\\
\footnotesize Faculty of Mathematics and Computer Science, University of Bucharest,\\
\footnotesize Academiei 14, 010014 Bucharest, Romania\\[1mm]
\footnotesize ${}^b$Simion Stoilow Institute of Mathematics of the Romanian Academy,\\
\footnotesize Calea Grivi\c tei 21, 010702 Bucharest, Romania\\[2mm]
\footnotesize E-mail: andrei.sipos@fmi.unibuc.ro\\
}
\date{}

\begin{document}

\maketitle

\begin{abstract}
Some quantitative results obtained by proof mining take the form of Herbrand disjunctions that may depend on additional parameters. We attempt to elucidate this fact through an extension to first-order arithmetic of the proof of Herbrand's theorem due to Gerhardy and Kohlenbach which uses the functional interpretation.

\noindent {\em Mathematics Subject Classification 2020}: 03F10, 03F05, 03F30.

\noindent {\em Keywords:} Herbrand's theorem, functional interpretation, first-order arithmetic, proof mining, metastability.
\end{abstract}

\section{Introduction}

The classical Herbrand theorem states, more or less, that for every first-order formula $\vp$ over a signature containing at least one constant symbol, such that $\vp$ has at most one free variable denoted by $x$ and such that classical first-order logic proves that $\exists x \vp$, there is a finite number of closed terms $t_1,\ldots,t_n$ over that signature -- extractable in some way from the proof -- such that it is a theorem of classical first-order logic that
$$\bigvee_{i=1}^n \left(\vp[x:=t_i]\right),$$
this kind of expression being usually called a {\it Herbrand disjunction}.

Such `term extraction' results have later been obtained for systems which include non-logical axioms, namely systems which may serve as a foundation of mathematics, like first-order arithmetic. We mention among such results Kreisel's no-counterexample interpretation \cite{Kre51,Kre52} and the various flavours of G\"odel's {\it Dialectica} \cite{God58}. The latter one, in a highly sophisticated form, plays nowadays a central role in the research program of {\it proof mining}, given maturity by the school of Kohlenbach (see e.g. his book \cite{Koh08} and his recent survey \cite{Koh19}), but originally initiated by Kreisel himself (under the name of {\it unwinding of proofs}), a program which aims to apply term extraction theorems to ordinary mathematical proofs in order to uncover information that may be not immediately apparent. One striking feature of this sort of theorems is that they generally do not extract terms expressible in the original system under discussion, but usually go beyond it in that they make use of concepts like higher-type functionals or recursion along large countable ordinals.

Despite this fact, it has been observed that in certain basic endeavours of proof mining, the extracted terms may take the form of a Herbrand disjunction ``of variable length'' \cite[p. 34]{Koh08g}. For example, Tao, in his work on multiple ergodic averages \cite{Tao08A}, rediscovered (the same example had been previously treated by Kreisel \cite[pp. 49--50]{Kre52}) the Herbrand normal form of the Cauchy property for bounded monotone sequences, a property which was then dubbed `metastability' \cite{Tao08} at the suggestion of Jennifer Chayes. The property is expressed as follows (denoting, for all $f: \N \to \N$ and for all $n \in \N$, by $\wt{f}(n)$ the quantity $n+f(n)$ and by $f^{(n)}$ the $n$-fold composition of $f$ with itself): if we take $k \in \N$, $g: \N \to \N$, and $(a_n)$ a, say, nonincreasing sequence in the interval $[0,1]$, then there is an $N$ such that for all $i$, $j \in [N,\wt{g}(N)]$, $|a_i-a_j| \leq 1/(k+1)$, and, moreover, $N$ can be taken to be an element of the finite sequence $0$, $\wt{g}(0),\ldots,\wt{g}^{(k)}(0)$. Thus, the conclusion of the statement before can be written as
$$\bigvee_{i=0}^{k} \left(\left(\forall i,\,j \in [x,\wt{g}(x)]\,|a_i-a_j| \leq \frac1{k+1}\right)\left[x:=\wt{g}^{\left(i \right)}(0)\right]\right),$$
so we have here a Herbrand disjunction which is {\bf variable} in $k$, but does not depend on the $g$, which only shows up in the terms themselves (a fact to which we shall later return).

Our goal here will be to attempt a logical explanation of this empirical fact.

Towards that end, we shall start with the proof\footnote{We mention that other higher-order treatments of Herbrand's theorem have recently been given by Ferreira and Ferreira \cite{FerFer17} -- using a `star-combinatory' calculus which has later seen arithmetical use in \cite{Fer20a,Fer20b} -- and by Afshari, Hetzl and Leigh \cite{AfsHetLei20} -- using higher-order recursion schemes.} of Herbrand's theorem due to Gerhardy and Kohlenbach \cite{GerKoh05} which uses the {\it Dialectica} interpretation, in particular a variant inspired by that of Shoenfield \cite{Sho67}, to construct witnesses that realize the interpreted formulas in a system similar to G\"odel's $\mathsf{T}$, but lacking recursors, and having case distinction functionals added in order to realize contraction. For a formula of the form $\exists x \vp$ as before, the extracted term is then $\beta$-reduced and it is shown that the resulting term has a sufficiently well-behaved form that one can read off it the classical Herbrand terms. What we do is to extend this proof to theories which are on the level of first-order arithmetic, dealing with the corresponding recursors (which one generally uses to interpret induction; see also \cite{Par72}) by using Tait's infinite terms \cite{Tai65}. This passage to the infinite allows us to prove in this extended context the corresponding version of the well-behavedness property mentioned before. Finally, we illustrate our result via the above metastability statement.

\section{Main results}

We shall denote by $\sigma_\ar$ the signature of arithmetic, containing the symbols $0$, $S$, $<$ having their usual natures and arities, together with symbols for all primitive recursive functions (in particular, the $+$ and $\cdot$ symbols; we shall need, though, to keep $<$ as a {\bf relation} symbol), and by $\mathcal{N}$ the standard $\sigma_\ar$-structure with universe $\N$ and with the symbols having their natural interpretation. We shall fix a signature $\sigma$ which contains $\sigma_\ar$.

We shall work in a formalization of first-order logic where the basic operators are $\neg$, $\lor$, and $\forall$, and we define, for any formulae $\vp$ and $\psi$, $\vp \to \psi:= \neg \vp \lor \psi$, $\vp \land \psi:= \neg(\neg\vp \lor \neg\psi)$, and $\exists x \vp:= \neg\forall x \neg\vp$. In addition,  if $t$ and $u$ are terms, by $t \leq u$ we shall mean $t<u \lor t = u$, and if $x$ is a variable, $t$ is a term and $\vp$ is a formula, we define the bounded quantification $\forall x\leq t \,\vp$ as $\forall x (x \leq t \to \vp)$ (i.e. $\forall x (\neg x \leq t \lor \vp)$). We do not define $\exists x \leq t \,\vp$ as $\exists x (x \leq t \land \vp)$ (i.e. $\neg\forall x \neg \neg (\neg x \leq t \lor \neg \vp)$) but rather as $\neg \forall x \leq t \neg \vp$ (i.e. $\neg \forall x (\neg x \leq t \lor \neg \vp)$), in order to more easily identify the bounded nature of the quantification. We shall call a formula {\it quantifier-free} if all its quantifiers are bounded, and {\it universal} ({\it existential}) if it results from a quantifier-free formula by prefixing it with a number of universal (respectively existential) quantifiers. We denote the set of free variables of a formula $\vp$ by $FV(\vp)$, and we shall maintain this notation for each kind of term or formula introduced below.

For the logical axioms and rules, we shall use, as in \cite{GerKoh05} (see also \cite{StrKoh08}) the system that Shoenfield used (as it can be seen on \cite[p. 215]{Sho67}, it is a modification of the Hilbert deduction system introduced earlier in that book for which see \cite[p. 21]{Sho67}) to state and prove the soundness theorems for his functional interpretation variant. The axioms are of the following forms:
\begin{itemize}
\item $\neg \vp \lor \vp$;
\item $\forall x \vp \to (\vp[x:=t])$ (where $x$ is a variable which is free for the term $t$ in $\vp$);
\item equality axioms,
\end{itemize}
while the rules have the following forms:
\begin{itemize}
\item $\vp \vdash \vp \lor \psi$;
\item $\vp \lor \vp \vdash \vp$ (the `contraction rule');
\item $(\vp \lor \psi) \lor \chi \vdash \vp \lor (\psi \lor \chi)$;
\item $\vp \lor \psi,\,\neg \vp \lor \chi \vdash \psi \lor \chi$;
\item $\vp \lor \psi \vdash (\forall x \vp) \lor \psi$ (where $x$ is a variable which is not free in $\psi$).
\end{itemize}

We shall work with the formalization of first-order arithmetic $\mathsf{PA^\sigma}$ (the $\sigma$ superscript indicating that the induction axiom schema below will work for any $\sigma$-formula) consisting of the following universal arithmetical sentences:
\begin{itemize}
\item $\forall x \neg (Sx=0)$;
\item $\forall x \forall y (Sx=Sy\to x=y)$;
\item $\forall x \neg (x<0)$;
\item $\forall x \forall y (x<Sy \leftrightarrow x \leq y)$,
\item defining axioms for the primitive recursive function symbols,
\end{itemize}
together with the induction axiom schema, which contains, for every $\sigma$-formula $\vp$, the formula
$$((\vp[x:=0]) \land \forall x(\vp \to (\vp[x:=Sx]))) \to \forall x \vp.$$
If we require, in the above, $\vp$ to be an existential formula, what we obtain is the schema of $\Sigma_1$-induction and the resulting theory is denoted by $\mathsf{I\Sigma_1^\sigma}$. With a moderate amount of effort, it can be shown that the following sentences are consequences of $\mathsf{PA^\sigma}$, and even of $\mathsf{I\Sigma_1^\sigma}$:
\begin{itemize}
\item $\forall x \forall y (x \leq y \lor y < x)$ (for this, see \cite[pp. 204--205]{Sho67});
\item $\forall x (\neg (x =  0) \to \exists y (Sy=x))$;
\item $\forall x\forall y (x \leq y \leftrightarrow \exists z(x+z=y))$.
\end{itemize}

We shall also fix a set $\Gamma$ of universal $\sigma$-sentences.

We shall now define the needed system of infinite terms in finite types. By a {\it finite type} we mean an element of the free algebra $\T$ with a unique binary operation $\to$ generated by one element denoted by $0$.  For example, $0$ and $(0 \to 0) \to 0$ are types. We call a finite sequence of types $\vec{\rho} = (\rho_1,\ldots,\rho_n) \in \T^*$ a {\it type-tuple}.

We extend the operation $\to$ to type-tuples, progressively, first to an operation $\to : \T^* \times \T \to \T^*$, recursively, by
$$() \to \rho := \rho, \quad (\vec{\theta}, \tau) \to \rho:= \vec{\theta} \to (\tau \to \rho),$$
then to an operation $\to : \T^* \times \T^* \to \T^*$, again recursively, by
$$\vec{\rho} \to () := (), \quad \vec{\rho} \to (\vec{\theta},\tau) := (\vec{\rho} \to \vec{\theta}, \vec{\rho} \to \tau).$$
For each $n \in \N$, we denote by $0^n$ the type-tuple $(0,\ldots,0)$ where $0$ is repeated $n$ times -- we remark that it is not difficult to show that $0^0 \to 0 = 0$ and that, for each $n$, $0^{n+1} \to 0 = 0 \to (0^n \to 0)$.

We shall now define $\omega$-terms recursively (technically, as below for $\omega$-formulas, it should be $(\omega,\sigma)$-terms, but we shall drop the fixed $\sigma$ as to not clutter the text). The terms are, as per Tait \cite{Tai65}, an extension of the simply-typed $\lambda$-calculus, and since $\lambda$-terms are an integral part of the calculus, more care has to be taken when defining substitution. The recursion goes as follows:
 
\begin{itemize}
\item we have a countable set of variables of each type $\rho$ denoted by $V_\rho$, and we put $V_0$ to be exactly the set of first-order variables;
\item for each function symbol $f$ of $\sigma$ of arity $n$ {\bf which is not a primitive recursive function symbol}, we have a corresponding constant term of type $0^n \to 0$ also denoted by $f$;
\item for each quantifier-free $\sigma$-formula $\vp$ {\bf which contains symbols not in $\sigma_\ar$} with $m$ free variables, we have a constant term $c_\vp$ of type $0^{m+2} \to 0$ (which shall serve as a case distinction functional);
\item if $x$ is a variable of type $\rho$ and $t$ is a term of type $\tau$, then $\lambda x.t$ is a term of type $\rho\to\tau$;
\item if $t$ is a term of type $\rho\to\tau$ and $s$ is a term of type $\rho$, then $ts$ is a term of type $\tau$;
\item if for each $n \in \N$, $t_n$ is a term of type $\rho$, then $(t_n)_{n \in \N}$ is a term of type $0 \to \rho$; if $\rho$ is $0$, we say that $(t_n)_{n \in \N}$ is {\it zero}, otherwise we say it is {\it non-zero}.
\end{itemize}

One can define for each $\omega$-term $t$ its length $|t|$ as an ordinal, in the following way:
\begin{itemize}
\item the length of a variable or constant is $1$;
\item the length of $\lambda x.t$ is $|t|+1$;
\item the length of $ts$ is $\max(|t|,|s|)+1$;
\item the length of $(t_n)_{n \in \N}$ is $\sup_{n \in \N} \left(\left|t_n\right|+1\right)$.
\end{itemize}

Having defined this length function, one can easily see that: (i) the $\omega$-terms form a set since one moment's reflection shows that already the set of those $\omega$-terms of length smaller than $\omega_1$ is closed under the constructors; (ii) one can prove properties of $\omega$-terms by structural induction.

If we have a type-tuple $\vec{\rho} = (\rho_1,\ldots,\rho_n) \in \T^*$, we say that a {\it term-tuple} of type-tuple $\vec{\rho}$ is a tuple of terms $\uu{t}=(t_1,...,t_n)$ such that for each $i$, $t_i$ is a term of type $\rho_i$. We denote this situation by $\tp(\uu{t}) = \vec{\rho}$. The application of terms can be extended recursively to term-tuples -- $t()$ will be $t$; $t(\uu{t'},u)$ will be $(t\uu{t'})u$; $()\uu{v}$ will be $()$; and, finally, $(\uu{t},u)\uu{v}$ will be $(\uu{t}\uu{v},u\uu{v})$. In addition, one shall build $\omega$-formulas from $\omega$-terms using the same constructors that have been used for $\sigma$-formulas in classical first-order logic -- note that the only equality relation which we shall admit is the one of type $0$.

We now introduce a number of abbreviations. First, we illustrate where recursors fit into this story. For each natural number $n$, we denote by $\uu{n}$ the term $S\ldots S0$ where $S$ is iterated $n$ times. If $\rho$ is a type, $a$ is a term of type $\rho$ and $b$ is a term of type $0\to(\rho\to\rho)$, we denote by $R_{a,b}$ the term $(t_n)_{n \in \N}$ of type $0\to\rho$ such that $t_0 = a$ and for every $n \in \N$, $t_{n+1} = b \uu{n} t_n$. We will also define {\it simultaneous recursors} (as in \cite[p. 48]{Koh08}) in the following way: if $\vec{\rho} = (\rho_1,\ldots,\rho_n)$ is a type-tuple, $\uu{a}$ is a term-tuple of type-tuple $\vec{\rho}$ and $\uu{b}$ is a term-tuple of type-tuple $0 \to (\vec{\rho} \to \vec{\rho})$, we denote by $R_{\uu{a},\uu{b}}$ the term-tuple of type-tuple $0 \to \vec{\rho}$ such that for each $i$, $(R_{\uu{a},\uu{b}})_i$ is the term $(t^i_n)_{n \in \N}$ of type $0\to\rho_i$ such that for every $i$, $t^i_0=a_i$ and for every $i$ and $n$, $t^i_{n+1} = \uu{b} \uu{n} t_n$.

Now that we have these recursors, we interrupt the defining of abbreviations in order to introduce some embedding functions denoted by $\iota$ -- one that takes $\sigma$-terms into $\omega$-terms by setting, for every variable $x$, $\iota(x):=x$ and, recursively, $\iota(f(t_1,\ldots,t_n)):=f(\iota(t_1))\ldots(\iota(t_n))$ if $f$ is not a primitive recursive function symbol, whereas if $f$ is a primitive recursive symbol, one defines the corresponding $\omega$-terms in the classical way using recursors and $\lambda$-abstractions; and another embedding function that takes $\sigma$-formulas into $\omega$-formulas in the natural way.

Resuming now the defining of abbreviations, using the primitive recursive decidability of quantifier-free formulas, for each quantifier-free $\sigma$-formula $\vp$ containing only symbols in $\sigma_\ar$ with $m$ free variables, we define a derived case distinction term $c_\vp$ of type $0^{m+2} \to 0$ in the established manner. Generally, {\bf we shall prefer} using these case distinction terms instead of the case distinction constants which we have introduced before for formulas involving symbols not in $\sigma_\ar$ essentially by replacing those terms in the formula at hand which involve those kind of symbols with variables, with the only non-arithmetic symbols which we will not be able to get rid of being the relation symbols.

If $\vp$ is a quantifier-free $\sigma$-formula with $m$ free variables $x_1,\ldots,x_m$, in order of their appearance in $\vp$, we define the term $a^i_\vp$ (which shall serve, by its definition using bounded search, as an argmax functional) as the term (where, here, and in the sequel, the hat will mean the omission of some element in a tuple)
$$\lambda x_1.\ldots\widehat{\lambda x_i}.\ldots.\lambda x_m. R_{0,\lambda v.\lambda w.(c_\vp x_1\ldots v\ldots x_mvw)},$$
where we note that $c_\vp$ can be a constant or not depending on the nature of $\vp$.

If $B$ is a quantifier-free $\omega$-formula with $q$ free variables $z_1,\ldots,z_q$, in order of their appearance in $B$, such that there is a $\sigma$-formula $\vp$ with $m$ free variables $x_1,\ldots,x_m$, in order of their appearance in $\vp$, and there are terms $t_1,\ldots,t_m$ with their free variables among $z_1,\ldots,z_q$ such that $B=\iota(\vp)[x_1:=t_1]\ldots[x_m:=t_m]$, we define $c_B$ as the term
$$\lambda z_1.\ldots.\lambda z_q.(c_\vp t_1\ldots t_m),$$
which, of course, depends on the decomposition chosen for $B$, but we can just choose in advance such a decomposition for each suitable formula $B$.
This allows us to define, if $i\in\{1,\ldots,q\}$ is such that $z_i$ is of type $0$, the term $a^i_B$ as the term
$$\lambda z_1.\ldots\widehat{\lambda z_i}.\ldots.\lambda z_q. R_{0,\lambda v.\lambda w.(c_B z_1\ldots v\ldots z_qvw)}.$$

We now proceed to define the semantics of this calculus. We put $\N_0:=\N$ and for any types $\rho$, $\theta$, we set $\N_{\rho \to \tau} := \N_\tau^{\N_\rho}$. Suppose we have a family of functions $e = (e_\rho)_{\rho \in \T}$ where for each $\rho$, $e_\rho :V_\rho \to \N_\rho$. Let $\mathcal{M}$ be a $\sigma$-structure whose $\sigma_\ar$-reduct is $\mathcal{N}$. We define the semantics of $\omega$-terms, relative to $\mathcal{M}$ and $e$, recursively:
\begin{itemize}
\item if $x$ is a variable of type $\rho$, then $x^{\mathcal{M},\omega}_e:=e_\rho(x)$;
\item for every $\omega$-term $f$ arising from a function symbol of $\sigma$ of arity $n$ also denoted by $f$, we set $f^{\mathcal{M},\omega}_e \in \N_{0^n \to 0}$ to be such that for every $a_1,\ldots,a_n \in \N$, we have that
$$f^{\mathcal{M},\omega}_e(a_1)\ldots(a_n) = f^\mathcal{M}(a_1,\ldots,a_n);$$
\item for every quantifier-free $\sigma$-formula $\vp$ which contains symbols not in $\sigma_\ar$ with $m$ free variables $x_1,\ldots,x_m$, in order of their appearance in $\vp$, we set $\left(c_\vp\right)^{\mathcal{M},\omega}_e \in \N_{0^{m+2} \to 0}$ to be such that for every $a_1,\ldots,a_m$, $b_1$, $b_2 \in \N$, we have that
$$\left(c_\vp\right)^{\mathcal{M},\omega}_e(a_1)\ldots(a_m)(b_1)(b_2)=
  \begin{cases} 
      \hfill b_1, \hfill &  \text{if } \|\vp\|^\mathcal{M}_{\left(e_0\right)_{x_1\leftarrow a_1,\ldots,x_m\leftarrow a_m}} = 1,\\
      \hfill b_2, \hfill & \text{otherwise}; \\
  \end{cases}$$
\item if $x$ is a variable of type $\rho$ and $t$ is a term of type $\tau$, then we set $(\lambda x.t)^{\mathcal{M},\omega}_e \in \N_{\rho\to\tau}$ to be such that for every $a \in \N_\rho$,
$$(\lambda x.t)^{\mathcal{M},\omega}_e(a) = t^{\mathcal{M},\omega}_{e_{x\leftarrow a}};$$
\item if $t$ is a term of type $\rho\to\tau$ and $s$ is a term of type $\rho$, then we set
$$(ts)^{\mathcal{M},\omega}_e := t^{\mathcal{M},\omega}_e\left(s^{\mathcal{M},\omega}_e\right) \in \N_\tau;$$
\item if for each $n \in \N$, $t_n$ is a term of type $\rho$, then  we set $\left((t_n)_{n \in \N}\right)^{\mathcal{M},\omega}_e \in \N_{0 \to \rho}= \N_\rho^\N$ to be such that for every $m \in \N$,
$$\left((t_n)_{n \in \N}\right)^{\mathcal{M},\omega}_e(m) = \left(t_m\right)^{\mathcal{M},\omega}_e.$$
\end{itemize}

The semantics of formulas is denoted by $\|\cdot\|^{\mathcal{M},\omega}_e$, or by $\models^\omega$ if the formula is a sentence, and is defined exactly as in the first-order case. One easily checks that for every $\sigma$-term $t$, $t^\mathcal{M}_{e_0} = (\iota(t))^{\mathcal{M},\omega}_e$, and that for every $\sigma$-formula $\vp$, $FV(\vp)=FV(\iota(\vp))$ and $\|\vp\|^\mathcal{M}_{e_0} = \|\iota(\vp)\|^{\mathcal{M},\omega}_e$, and thus that if $\vp$ is a sentence, then $\mathcal{M} \models \vp$ iff $\mathcal{M} \models^\omega \iota(\vp)$ (the last one by extending first-order valuations by zero to the higher types).

We will now express the fact that the derived terms have the natural intended semantics. We have that (using here that the $\sigma_\ar$-reduct of $\mathcal{M}$ is $\mathcal{N}$):
\begin{itemize}
\item for each $n \in \N$, $(\uu{n})^{\mathcal{M},\omega}_e = n$;
\item for each $a$ of type $\rho$ and $b$ of type $0 \to (\rho \to \rho)$,
$$(R_{a,b})^{\mathcal{M},\omega}_e(0) = a^{\mathcal{M},\omega}_e,$$
and, for each $n \in \N$
$$(R_{a,b})^{\mathcal{M},\omega}_e(n+1) = b^{\mathcal{M},\omega}_e(n)((R_{a,b})^{\mathcal{M},\omega}_e(n));$$
\item for each type-tuple $\vec{\rho} = (\rho_1,\ldots,\rho_n)$, for each $\uu{a}$ of type-tuple $\vec{\rho}$ and $\uu{b}$ of type-tuple $0 \to (\vec{\rho} \to \vec{\rho})$, and for each $i$,
$$((R_{\uu{a},\uu{b}})_i)^{\mathcal{M},\omega}_e(0) = (a_i)^{\mathcal{M},\omega}_e,$$
and, for each $n \in \N$,
$$((R_{\uu{a},\uu{b}})_i)^{\mathcal{M},\omega}_e(n+1) = (b_i)^{\mathcal{M},\omega}_e(n)(((R_{\uu{a},\uu{b}})_1)^{\mathcal{M},\omega}_e(n))\ldots (((R_{\uu{a},\uu{b}})_l)^{\mathcal{M},\omega}_e(n));$$
\item if $\vp$ is a quantifier-free $\sigma$-formula containing only symbols in $\sigma_\ar$ with $m$ free variables $x_1,\ldots,x_m$, in order of their appearance in $\vp$, then, for each $a_1,\ldots,a_m$, $b_1$, $b_2 \in \N$, we have that
$$\left(c_\vp\right)^{\mathcal{M},\omega}_e(a_1)\ldots(a_m)(b_1)(b_2)=
  \begin{cases} 
      \hfill b_1, \hfill &  \text{if } \|\vp\|^\mathcal{M}_{\left(e_0\right)_{x_1\leftarrow a_1,\ldots,x_m\leftarrow a_m}} = 1,\\
      \hfill b_2, \hfill & \text{otherwise}; \\
  \end{cases}$$
\item if $\vp$ is a quantifier-free $\sigma$-formula with $m$ free variables $x_1,\ldots,x_m$, in order of their appearance in $\vp$, then for each $a_1,\ldots,\widehat{a_i},\ldots,a_m\in \N$, if we set, for each $n\in \N$,
$$N(n):= \{ p<n \mid \mathcal{M} \models \vp[x_1:=\uu{a_1}]\ldots\widehat{[x_i:=\uu{a_i}]}\ldots[x_m:=\uu{a_m}][x_i:=\uu{p}] \},$$
then, for each $n$,
$$\left(a^i_\vp\right)^{\mathcal{M},\omega}_e(a_1)\ldots\widehat{(a_i)}\ldots (a_m)(n) = 
\begin{cases} 
      \hfill \max N(n), \hfill &  \text{if } N(n) \neq \emptyset,\\
      \hfill 0, \hfill & \text{otherwise}. \\
  \end{cases}$$
\item if $B$ is a quantifier-free $\omega$-formula with $q$ free variables $z_1,\ldots,z_q$, in order of their appearance in $B$, such that there is a $\sigma$-formula $\vp$ with $m$ free variables $x_1,\ldots,x_m$, in order of their appearance in $\vp$, and there are terms $t_1,\ldots,t_m$ with their free variables among $z_1,\ldots,z_q$ such that $B=\iota(\vp)[x_1:=t_1]\ldots[x_m:=t_m]$, then for any suitable $a_1,\ldots,a_q$, and any $b_1$, $b_2 \in \N$, we have that
$$\left(c_B\right)^{\mathcal{M},\omega}_e(a_1)\ldots(a_q)(b_1)(b_2)=
 \begin{cases} 
      \hfill b_1, \hfill &  \text{if } \|B\|^\mathcal{M,\omega}_{e_{z_1\leftarrow a_1,\ldots,z_q\leftarrow a_q}} = 1,\\
      \hfill b_2, \hfill & \text{otherwise}; \\
  \end{cases}$$
\item in the hypothesis above, if $i\in\{1,\ldots,q\}$ is such that $z_i$ is of type $0$, then for any suitable $a_1,\ldots,\widehat{a_i},\ldots,a_m$, if we set, for each $n\in \N$,
$$N(n):= \{ p<n \mid \mathcal{M} \models B[z_1:=\uu{a_1}]\ldots\widehat{[z_i:=\uu{a_i}]}\ldots[z_q:=\uu{a_q}][z_i:=\uu{p}] \},$$
then, for each $n$,
$$\left(a^i_B\right)^{\mathcal{M},\omega}_e(a_1)\ldots\widehat{(a_i)}\ldots (a_q)(n) = 
\begin{cases} 
      \hfill \max N(n), \hfill &  \text{if } N(n) \neq \emptyset,\\
      \hfill 0, \hfill & \text{otherwise}. \\
  \end{cases}$$
\end{itemize}

The main tool we shall use is, as announced in the Introduction, Shoenfield's variant of G\"odel's {\it Dialectica} interpretation \cite[pp. 214--222]{Sho67}. This associates to every first-order $\sigma$-formula $\vp$ a formula $\vp_\Sh$ in the calculus we have defined above, together with two disjoint term-tuples of variables $\uu{u}_\vp$ and $\uu{x}_\vp$, recursively, in the following way (with appropriate renamings of variables at crucial points):
\begin{itemize}
\item if $\vp$ is atomic, then $\vp_\Sh=\iota(\vp)$ and $\uu{u}_\vp=\uu{x}_\vp=()$;
\item if $\vp$ is $\neg\psi$, then, taking $\uu{f}$ to be a term-tuple of fresh variables of type-tuple $\tp(\uu{u}_\psi) \to \tp(\uu{x}_\psi)$, $\uu{u}_\vp = \uu{f}$, $\uu{x}_\vp = \uu{u}_\psi$ and $\vp_\Sh = \neg \psi_\Sh [\uu{x}_\psi := \uu{f} \uu{u}_\psi]$;
\item if $\vp$ is $\psi \lor \chi$, then $\vp_\Sh = \psi_\Sh \lor \chi_\Sh$, $\uu{u}_\vp$ is $\uu{u}_\psi$ concatenated with $\uu{u}_\chi$ and $\uu{x}_\vp$ is $\uu{x}_\psi$ concatenated with $\uu{x}_\chi$;
\item if $\vp$ is $\forall z \psi$, where the quantification is not bounded, then $\vp_\Sh=\psi_\Sh$, $\uu{u}_\vp$ is $(z)$ concatenated with $\uu{u}_\psi$ and $\uu{x}_\vp=\uu{x}_\psi$;
\item  if $\vp$ is $\forall z \leq t\, \psi$, then $\vp_\Sh=\forall z \leq t\,\psi_\Sh$, $\uu{u}_\vp=\uu{u}_\psi$ and $\uu{x}_\vp=\uu{x}_\psi$.
\end{itemize}

One remarks that if $A$ is a quantifier-free formula and $\uu{z}$ is a term-tuple of variables, then:
\begin{itemize}
\item $A_\Sh=\iota(A)$, $\uu{u}_A=\uu{x}_A=()$;
\item $\left(\forall \uu{z} A\right)_\Sh=\iota(A)$, $\uu{u}_{\forall \uu{z} A}=\uu{z}$, $\uu{x}_{\forall \uu{z} A}=()$ (if the added quantifications are not bounded);
\item $\left(\exists \uu{z} A\right)_\Sh=\neg\neg\iota(A) \sim \iota(A)$, $\uu{u}_{\exists \uu{z} A}=()$, $\uu{x}_{\exists \uu{z} A}=\uu{z}$ (ditto).
\end{itemize}

The following result justifies the definition of the interpretation.

\begin{theorem}[{Soundness theorem}]\label{sound}
Let $\vp$ be a $\sigma$-formula such that
$$\mathsf{PA^\sigma} + \Gamma \vdash \vp.$$
Then there is a term-tuple $\uu{t}$ with $\tp\left(\uu{t}\right)=\tp\left(\uu{u}_\vp\right) \to \tp\left(\uu{x}_\vp\right)$ and $FV\left(\uu{t}\right)\subseteq FV(\vp)$, extractable from the proof, such that for every $\sigma$-structure $\mathcal{M}$ such that its $\sigma_\ar$-reduct is $\mathcal{N}$ and $\mathcal{M} \models \Gamma$ and for every higher-order valuation $e$, we have that
$$\left\|\forall \uu{u}_\vp \left(\vp_\Sh \left[\uu{x}_\vp := \uu{t}\, \uu{u}_\vp\right]\right) \right\|^{\mathcal{M},\omega}_e = 1.$$
\end{theorem}

\begin{proof}
We have to extend the proof given by \cite[pp. 220--222]{Sho67} and \cite[Lemma 8]{GerKoh05} to cover our situation. We remark that the equality axioms, the sentences in $\Gamma$ and the non-induction axioms of $\mathsf{PA^\sigma}$ are all universal and thus trivially interpreted. In addition, it is easily checked that the axiom and rule regarding the universal quantifier still get witnesses given our interpretation of bounded quantification.

We shall now show how the induction axiom schema may be interpreted directly, since \cite{Sho67} only shows how to interpret the induction rule from which the axiom schema may be derived using contraction, and hence we shall need to use case distinction functionals. We follow the ideas of Parsons \cite[Section 4]{Par72}. Consider an instance of the induction axiom
$$\psi := ((\vp[x:=0]) \land \forall x(\vp \to (\vp[x:=Sx]))) \to \forall x \vp.$$
We may consider $\vp$ to be of the form $\exists z P(x,z)$ where $P$ is a relation symbol of $\sigma$ which is not in $\sigma_\ar$. (The complications arising from the general case, where one introduces multiple, possibly alternately, quantified variables, are that one shall need simultaneous recursors, not just of type $0$, and appropriate interleavings of term-tuples.)

We obtain that $\uu{u}_\psi = (u,f,y)$ with $\tp(\uu{u}_\psi) = (0,0\to 0,0)$, $\uu{x}_\psi = (x,z,v)$ with $\tp(\uu{x}_\psi)=(0,0,0)$ and
$$\psi_\Sh \sim P(0,u) \land ((P(x,z) \to P(Sx, fxz))) \to P(y,v).$$

Therefore, one must, roughly speaking, write $x$, $z$, $v$ in terms of $u$, $f$, $y$ such that the above formula holds. Set
$$B:=P(w,R_{u,f}w) \land \neg P(Sw,R_{u,f}(Sw)),$$
which is clearly decomposable. Then one can take
$$x:= a^1_Bufy,\quad z:= R_{u,f}x, \quad v:=R_{u,f}y,$$
and after a quick check we are done.
\end{proof}

The following is the analogue of \cite[Lemma 8]{GerKoh05}.

\begin{lemma}\label{l1}
Let $A$ be a quantifier-free $\sigma$-formula such that it has at most one free variable denoted by $x$ and such that the quantification in $\exists x A$ is not bounded. Assume that
$$\mathsf{PA^\sigma} + \Gamma \vdash \exists x A.$$
Then there is a closed $\omega$-term of type $0$, extractable from the proof, such that for every $\sigma$-structure $\mathcal{M}$ such that its $\sigma_\ar$-reduct is $\mathcal{N}$ and $\mathcal{M} \models \Gamma$, we have that
$$\mathcal{M} \models^\omega \iota(A)[x:=t].$$
\end{lemma}

\begin{proof}
We have that $(\exists x A)_\Sh \sim \iota(A)$, $\uu{u}_{\exists x A}=()$, $\uu{x}_{\exists x A}=(x)$, so $\tp\left(\uu{u}_{\exists x A}\right) \to \tp\left(\uu{x}_{\exists x A}\right) = (0)$. Thus, by Theorem~\ref{sound}, there is an $\omega$-term $t$ of type $0$ with $FV(t)\subseteq FV(\exists x A) = \emptyset$ (so $t$ is closed), extractable from the proof, such that for every $\sigma$-structure $\mathcal{M}$ such that its $\sigma_\ar$-reduct is $\mathcal{N}$ and $\mathcal{M} \models \Gamma$ and for every higher-order valuation $e$, we have that
$$\left\|\iota(A) \left[x:=t\right] \right\|^{\mathcal{M},\omega}_e = 1,$$
which, since $\iota(A) \left[x:=t\right]$ is a sentence, may be expressed as
$$\mathcal{M} \models^\omega \iota(A)[x:=t],$$
which is what we needed to show.
\end{proof}

If we again follow Tait \cite{Tai65} and consider the term rewriting system on the $\omega$-terms generated by the reduction relations $(\lambda x.t)(s) \leadsto t[x:=s]$, $\left((t_n)_{n \in \N}r\right)s\leadsto \left((t_ns)_{n\in\N}\right)r$ and $(t_n)_{n \in \N}\uu{m} \leadsto t_m$, as extended by compatibility with the term constructors, we have that each term $t$ has a normal form $s$ having the same free variables, such that for every $\sigma$-structure $\mathcal{M}$ whose $\sigma_\ar$-reduct is $\mathcal{N}$, and every valuation $e$, $\|t=s\|^{\mathcal{M},\omega}_e=1$.

We shall now work towards proving the analogue of \cite[Lemma 10]{GerKoh05}.

\begin{lemma}\label{la1}
Let $a$ and $b$ be $\omega$-terms such that $ab$ is a closed term in normal form. Then either the term $a$ is a non-zero $(t_n)_{n \in \N}$ or there is an $n \in \N$ such that the term $ab$ is of type $0^n \to 0$.
\end{lemma}

\begin{proof}
We prove the conclusion by induction on the length of $ab$. Since $ab$ is closed, $a$ cannot be a variable. If $a$ is a non-zero $(t_n)_{n \in \N}$, we are done. If $a$ is a constant or is a zero $(t_n)_{n \in \N}$, then there is a $k$ such that $a$ is of a type $0^k\to 0$, so $ab$ is of type $0^{k-1}\to0$. Since $ab$ is in normal form, $a$ cannot be a $\lambda$-expression. Therefore, $a$ is of the form $uv$. But then $uv$ is also a closed term in normal form, so, by the induction hypothesis, either $u$ is a non-zero $(t_n)_{n \in \N}$, contradicting the fact that $(uv)b$ is in normal form, or there is a $k$ such that $a$ is of a type $0^k\to 0$, so $ab$ is of type $0^{k-1}\to0$.
\end{proof}

\begin{lemma}\label{la2}
Let $s$ be a closed $\omega$-term of type $0$ in normal form. Then the term $s$ contains neither $\lambda$-expressions, nor non-zero $(t_n)_{n \in \N}$'s  -- thus, $s$ is built solely out of constants, zero $(t_n)_{n \in \N}$'s and applications.
\end{lemma}

\begin{proof}
Assume that $s$ contains a $\lambda$-expression or a non-zero $(t_n)_{n \in \N}$ and let $r$ be such a subterm which is not contained in another such, and also such that if it is contained in a term $q$ such that $pq$ is a subterm of $s$, $p$ does not contain $\lambda$-expressions or non-zero $(t_n)_{n \in \N}$'s (this can be specified by structural recursion).

Suppose first that $r$ is a $\lambda$-expression, and denote it by $\lambda x.u$. We cannot have that $s=\lambda x.u$, since $s$ is of type $0$. Therefore, there is a subterm of $s$ which lies immediately above $\lambda x.u$, which cannot be, by our assumption, a non-zero $(t_n)_{n \in \N}$, neither can it be a zero $(t_n)_{n \in \N}$ (since then each $t_n$ is of type $0$, while $\lambda x.u$ is not), and also neither can it be of the form $(\lambda x.u)v$ (since that would contradict that $s$ is in normal form). Therefore, that subterm is of the form $v(\lambda x.u)$. Now, if $v$ were a variable, and since $s$ is closed, then $v(\lambda x.u)$ would have to lie inside another $\lambda$-expression, contradicting our assumption. Since $\lambda x.u$ is not of type $0$, $v$ cannot be a constant or a term of the form $(t_n)_{n \in \N}$, and also it cannot be a $\lambda$-expression since that would contradict that $s$ is in normal form. Therefore, $v$ is of the form $ab$. Since $(ab)(\lambda x.u)$ is not contained in another $\lambda$-expression and $s$ is closed, we have that $ab$ is closed. Also, $ab$ is in normal form. Therefore, by Lemma~\ref{la1}, either the term $a$ is a non-zero $(t_n)_{n \in \N}$, which contradicts the fact that $(ab)(\lambda x.u)$ is in normal form (or already contradicts our assumption), or there is an $n$ such that $ab$ is of type $0^n \to 0$, which is impossible, since $\lambda x.u$ is not of type $0$.

Suppose now that $r$ is a non-zero $(t_n)_{n \in \N}$. We cannot have that $s=(t_n)_{n \in \N}$, since $s$ is of type $0$. Therefore, there is a subterm of $s$ which lies immediately above $(t_n)_{n \in \N}$, which, by our assumption, is either of the form $u(t_n)_{n \in \N}$ or $(t_n)_{n \in \N}u$. In the first case, again by our assumption that $u$ cannot contain $\lambda$-expressions or non-zero infinite terms, and also it cannot contain variables since then $u(t_n)_{n \in \N}$ would have to be contained in a $\lambda$-expression, contradicting our assumption, there is an $n \in \N$ such that $u$ is of the type $0^n\to0$, since this is the only kind of type one can build out of constants, zero $(t_n)_{n \in \N}$'s and applications, and this would imply that $(t_n)_{n \in \N}$ is of type $0$, which is a contradiction. Thus, the subterm must be of the form $(t_n)_{n \in \N}u$, which is not of type $0$, since $(t_n)_{n \in \N}$ is non-zero, so it cannot be the whole of $s$. We can therefore repeat the argument before to derive that this subterm must be in its turn contained in a larger one $\left((t_n)_{n \in \N}u\right)v$, which is not in normal form, and this contradiction finishes our proof.
\end{proof}

\begin{lemma}\label{la3}
Let $s$ be a closed $\omega$-term of type $0$ in normal form and let $u$ be a subterm of $s$. Then there is a $k \in \N$ and terms $s_1,\ldots,s_k$ of type $0$ such that there is a term $f$ which is either a constant or a zero $(t_n)_{n \in \N}$ with $u=fs_1\ldots s_k$ (and thus $s_1,\ldots,s_k$ are also in normal form).
\end{lemma}

\begin{proof}
We prove the conclusion by induction on the length of $u$, using what we learned in Lemma~\ref{la2}. If $u$ is a constant or is a zero $(t_n)_{n \in \N}$, the conclusion holds (we take $k:=0$ and $f:=u$). If $u$ is of the form $vw$, then, by the induction hypothesis, there is an $l \in \N$ and terms $s_1,\ldots,s_l$ of type $0$ such that there is a term $f$ which is either a constant or a zero $(t_n)_{n \in \N}$ with $v=fs_1\ldots s_l$. Then $w$ is of type $0$, so we may take $k:=l+1$ and $s_k:=w$.
\end{proof}

Based on the characterization in Lemma~\ref{la3}, we shall now define, for each $\sigma$-structure $\mathcal{M}$ whose $\sigma_\ar$-reduct is $\mathcal{N}$, and for each closed $\omega$-term $s$ of type $0$ in normal form, a finite set of closed $\sigma$-terms denoted by $T^\mathcal{M}_s$, recursively in the length of $s$:
\begin{itemize}
\item if $f$ is a function symbol of $\sigma$ of arity $n$ and $s$ is of the form $fs_1\ldots s_n$, we put
$$T^\mathcal{M}_s:= \{f(t_1,\ldots,t_n) \mid \text{for all $i$, }t_i \in T^\mathcal{M}_{s_i} \};$$
\item if $\vp$ is a quantifier-free $\sigma$-formula $\vp$ which contains symbols not in $\sigma_\ar$ with $m$ free variables and $s$ is of the form $c_\vp s_1\ldots s_m s_{m+1}s_{m+2}$, we put
$$T^\mathcal{M}_s:= T^\mathcal{M}_{s_{m+1}} \cup T^\mathcal{M}_{s_{m+2}};$$
\item if for each $n \in \N$, $t_n$ is a term of type $0$ and $s$ is of the form $\left((t_n)_{n \in \N}\right)s_1$, we put
$$T^\mathcal{M}_s:= \bigcup_{r \in T^\mathcal{M}_{s_1}} T^\mathcal{M}_{t_{r^\mathcal{M}}}.$$
\end{itemize}

\begin{lemma}\label{l3}
Let $s$ be a closed $\omega$-term of type $0$ in normal form and let $\mathcal{M}$ be a $\sigma$-structure whose $\sigma_\ar$-reduct is $\mathcal{N}$. Then there is a finite set $S$ of closed first-order $\sigma$-terms such that
$$\mathcal{M} \models^\omega \bigvee_{r \in S} s = \iota(r).$$
\end{lemma}

\begin{proof}
We just take $S:=T^\mathcal{M}_s$.
\end{proof}

We can now express our main result.

\begin{theorem}\label{main}
Let $A$ be a quantifier-free $\sigma$-formula such that it has at most one free variable denoted by $x$ and such that the quantification in $\exists x A$ is not bounded. Assume that
$$\mathsf{PA^\sigma} + \Gamma \vdash \exists x A.$$
Let $\mathcal{M}$ be a $\sigma$-structure whose $\sigma_\ar$-reduct is $\mathcal{N}$ and $\mathcal{M} \models \Gamma$. Then Shoenfield's extraction algorithm can be used to explicitly exhibit from the given proof a finite set $S$ of closed first-order $\sigma$-terms such that
$$\mathcal{M} \models \bigvee_{r \in S} A[x:=r].$$
\end{theorem}

\begin{proof}
By Lemma~\ref{l1}, there is a closed $\omega$-term of type $0$ such that
$$\mathcal{M} \models^\omega \iota(A)[x:=t].$$
Let $s$ be the closed $\omega$-term which is the normal form of $t$. Since $\mathcal{M} \models^\omega t=s$, we have that
$$\mathcal{M} \models^\omega \iota(A)[x:=s].$$
By Lemma~\ref{l3}, we get that there is a finite set $S$ of closed first-order $\sigma$-terms such that
$$\mathcal{M} \models^\omega \bigvee_{r \in S} s = \iota(r),$$
so
$$\mathcal{M} \models^\omega \bigvee_{r \in S} \iota(A)[x:=\iota(r)],$$
which can be written as
$$\mathcal{M} \models^\omega \iota\left(\bigvee_{r \in S} A[x:=r]\right),$$
from which we get that
$$\mathcal{M} \models \bigvee_{r \in S} A[x:=r],$$
i.e. what we needed to show.
\end{proof}

Of course, the plain existence of $S$ in the statement of Theorem~\ref{main}, as that in Lemma~\ref{l3} before it, is completely trivial as it stands, since we can just take an appropriate witness $n \in \N$ in $\mathcal{M}$ and take $S:=\{\uu{n}\}$. The real content of our result lies in the way $S$ is constructed, since one does not necessarily ``call'' the whole of $\mathcal{M}$ in the course of the recursive definition, so one may hope to recover some amount of uniformity in $\mathcal{M}$, i.e. one might not care about the interpretations of some of the symbols {\bf not} in $\sigma_\ar$; we briefly note that if the proof takes place entirely within $\sigma_\ar$, then the case distinction constants never show up and the resulting set is necessarily a singleton, a situation which is not of interest to us here. All of this could be made a bit more precise, but we prefer to give an example to illustrate these facts to which we also alluded in the Introduction, and this we do in the next section.

\section{A concrete example}

We shall now show how our results can elucidate the example given as motivation in the Introduction. In the first step, we shall spell out the non-quantitative proof in detail in order to help with the formalization and the extraction.

Let, therefore, $k \in \N$, $g:\N \to \N$ and $(a_n)_{n \in \N} \subseteq [0,1]$ be nonincreasing. We want to show that there is an $N \in \N$ such that for all $i$, $j \in [N,\wt{g}(N)]$,
$$|a_i-a_j| \leq \frac1{k+1}.$$
It is enough to show that there is an $N$ such that
$$a_N - a_{\wt{g}(N)} \leq \frac1{k+1},$$
since, then, taking the same $N$ and $i$, $j \in [N,\wt{g}(N)]$ and assuming w.l.o.g. that $i\leq j$, we have that
$$|a_i-a_j| = a_i - a_j \leq a_N - a_{\wt{g}(N)} \leq \frac1{k+1}.$$

We shall follow now the argument from \cite[Lemma 6.3]{KohLeu10}. Assume towards a contradiction that, for all $N \in \N$,
$$a_N - a_{\wt{g}(N)} > \frac1{k+1}.$$
We now show that for all $x \in \N$ there is an $y \in \N$ such that
$$a_0 - a_y > \frac{x+1}{k+1}.$$
We prove this by induction on $x$. For $x=0$, we have that
$$a_0 - a_{\wt{g}(0)} > \frac1{k+1},$$
so we can take $y:=\wt{g}(0)$.
Let now $x \in \N$ and assume that there is a $y$ such that
$$a_0 - a_y > \frac{x+1}{k+1}.$$
We want to show that there is a $w$ such that
$$a_0 - a_w > \frac{x+2}{k+1}.$$
Since 
$$a_y - a_{\wt{g}(y)} > \frac1{k+1},$$
we have that
$$a_0 - a_{\wt{g}(y)} = a_0 - a_y + a_y - a_{\wt{g}(y)} > \frac{x+1}{k+1} + \frac1{k+1} = \frac{x+2}{k+1},$$
so we can take $w:=\wt{g}(y)$ and the induction is finished.
Now, if we take $x:=k$, we have that there is an $y \in \N$ such that
$$a_0 - a_y > \frac{k+1}{k+1} =1,$$
but $a_0 - a_y \leq a_0 \leq 1$, which yields a contradiction, and thus we are finished with the proof.

To formalize the above argument in our framework, we shall construct the first-order signature $\sigma$ by adjoining to $\sigma_\ar$ the function symbols $k$ and $g$ of arity (obviously) $0$ and $1$, respectively, together with a relation symbol $P$ of arity $3$ such that $P(v,w,l)$ signifies
$$a_v - a_w \leq \frac{l}{k+1}.$$
It is then clear that if we take
$$\Gamma:= \{\forall v\forall p\forall r\forall b((\neg P(v,p,b) \land \neg P(p,r,S0)) \to \neg P(v,r,Sb)), \forall t P(0,t,Sk)\},$$
which is a set of universal $\sigma$-sentences, then the whole of the above argument is formalizable in $\mathsf{I\Sigma_1^\sigma} + \Gamma$.

We shall now present the result after it has passed through the interpretation, being a bit sloppy with the distinction between terms and their meaning in a model, a distinction that we shall reinstate shortly afterwards.

Assume one sets, for arbitrary $k \in \N$ and $g:\N \to \N$,
$$B := \neg P(0,R_{u,f}z,Sz) \land P(0,R_{u,f}(Sz),SSz), \quad a := a^3_B(\wt{g}0)(\lambda v.\lambda w.\wt{g}w)k,$$
$$q := R_{\wt{g}0,\lambda v.\lambda w.\wt{g}w}a, \quad N := c_{P(0,\wt{g}0,S0)}0q.$$
Then this $N$ is the extracted witness, in the sense that $\Gamma$ implies $P(N,gN,S0)$.

Let us now verify this. Assume towards a contradiction that $\neg P(N,\wt{g}N,S0)$. Assuming in addition $P(0,\wt{g}0,S0)$, then $N=0$, therefore $\neg P(N,\wt{g}N,S0)$, a contradiction. Therefore $\neg P(0,\wt{g}0,S0)$, from which we get $N=q$ and thus $\neg P(q,\wt{g}q,S0)$.

We now show that for all $x \leq k$, $\neg P(0,R_{\wt{g}0,\lambda v.\lambda w.\wt{g}w}x,Sx)$, by induction on $x$.

For $x=0$, we have that $R_{\wt{g}0,\lambda v.\lambda w.\wt{g}w}x=\wt{g}0$, so we have to show that $\neg P(0,\wt{g}0,S0)$, which we know.

We now have to show that for all $x<k$, $\neg P(0,R_{\wt{g}0,\lambda v.\lambda w.\wt{g}y}x,Sx)$ implies $\neg P(0,R_{\wt{g}0,\lambda v.\lambda w.\wt{g}y}(Sx),SSx)$. Assume this wasn't so and take $x$ minimal with this property. Then $x=a$ (by the definition of $a$), and thus $R_{\wt{g}0,\lambda v.\lambda w.\wt{g}w}x=q$ and
$$R_{\wt{g}0,\lambda v.\lambda w.\wt{g}w}(Sx) = (\lambda v.\lambda w.\wt{g}w)a(R_{g0,\lambda v.\lambda w.\wt{g}w}a) = (\lambda v.\lambda w.\wt{g}w)aq= \wt{g}q.$$
We thus have $\neg P(0,q,Sa)$ and $P(0,\wt{g}q,SSa)$. But, since $\neg P(q,\wt{g}q,S0)$, by the first axiom in $\Gamma$ we get that $\neg P(0,\wt{g}q,SSa)$, which yields a contradiction that finishes the induction.

By setting $x:=k$, we get that $\neg P(0,R_{\wt{g}0,\lambda v.\lambda w.\wt{g}w}k,Sk)$. But, by the second axiom in $\Gamma$, we have that $P(0,R_{\wt{g}0,\lambda v.\lambda w.\wt{g}w}k,Sk)$, which again yields a contradiction, and we are done with the verification.

Let us now consider formally the terms set above. We seek to compute a normal form for $N$, warning the reader that some (hopefully trivial) preparatory normalizing has already been done in the expression of $N$. We first remark that $r':=\left(\wt{g}^{(n+1)}0\right)_{n \in \N}$ is a normal form for $R_{\wt{g}0,\lambda v.\lambda w.\wt{g}w}$. Set
$$\vp:= \neg P(a,b,c) \land P(d,e,f).$$

Thus, if we put $u_0 := 0$ and for all $n \in \N$,
$$u_{n+1}:=c_\vp0\left(r'\uu{n}\right)(\uu{n+1})0\left(r'\uu{n+1}\right)(\uu{n+2})\uu{n} u_n,$$
then $a':=(u_n)_{n \in \N}k$ is a normal form for $a$. We also get that $q':=\left(\wt{g}^{(n+1)}0\right)_{n \in \N}a'$ is a normal form for $q$ and $N':=c_{P(0,\wt{g}0, S0)}0q'$ is a normal form for $N$.

Now take $\mathcal{M}$ to be a $\sigma$-structure whose $\sigma_\ar$-reduct is $\mathcal{N}$. We see that
$$T^\mathcal{M}_{u_0} = T^\mathcal{M}_0 =\{0\}$$
and, for all $n \in \N$,
$$T^\mathcal{M}_{u_{n+1}} = T^\mathcal{M}_{c_\vp0\left(r'\uu{n}\right)(\uu{n+1})0\left(r'\uu{n+1}\right)(\uu{n+2})\uu{n} u_n} = T^\mathcal{M}_{\uu{n}} \cup T^\mathcal{M}_{u_n} = \{\uu{n}\} \cup T^\mathcal{M}_{u_n},$$
so, for all $n \in \N$,
$$T^\mathcal{M}_{u_n} = \{\uu{i} \mid 0 \leq i \leq \max(n-1,0)\}.$$
We may now compute:
\begin{align*}
T^\mathcal{M}_{a'}&= \bigcup_{r \in T^\mathcal{M}_k} T^\mathcal{M}_{u_{r^\mathcal{M}}} = \bigcup_{r \in \{k\}} T^\mathcal{M}_{u_{r^\mathcal{M}}} = T^\mathcal{M}_{u_{k^\mathcal{M}}} = \left\{\uu{i} \mid 0 \leq i \leq \max\left(k^\mathcal{M}-1,0\right)\right\}.\\
T^\mathcal{M}_{q'}&= \bigcup_{r \in T^\mathcal{M}_{a'}} T^\mathcal{M}_{\wt{g}^{\left(r^\mathcal{M} + 1\right)}0} = \bigcup_{r \in T^\mathcal{M}_{a'}} \left\{\wt{g}^{\left(r^\mathcal{M} + 1\right)}0\right\} = \left\{ \wt{g}^{\left(r^\mathcal{M} + 1\right)}0 \mid r \in T^\mathcal{M}_{a'} \right\} \\
&= \left\{ \wt{g}^{(i+1)}0 \mid0 \leq i \leq \max\left(k^\mathcal{M}-1,0\right) \right\} = \left\{ \wt{g}^{(i)}0 \mid1 \leq i \leq \max\left(k^\mathcal{M},1\right) \right\}.\\
T^\mathcal{M}_{N'}&= T^\mathcal{M}_0 \cup T^\mathcal{M}_{q'} = \{0\} \cup \left\{ \wt{g}^{(i)}0 \mid1 \leq i \leq \max\left(k^\mathcal{M},1\right) \right\} = \left\{ \wt{g}^{(i)}0 \mid0 \leq i \leq \max\left(k^\mathcal{M},1\right) \right\}.
\end{align*}
Therefore, we have recovered a Herbrand disjunction of the form exhibited in the Introduction (and also in \cite[Lemma 6.3]{KohLeu10}), and we can clearly see that the computed set depends on $k^\mathcal{M}$ but not on $g^\mathcal{M}$ or $P^\mathcal{M}$, and thus we obtain uniformity in each class of $\sigma$-structures whose $\sigma_\ar$-reduct is $\mathcal{N}$, which satisfy $\Gamma$, and which interpret $k$ the same way.

\section{Acknowledgements}

I would like to thank Ulrich Kohlenbach for originally suggesting to me to look at the infinitary calculus which was introduced by Tait in \cite{Tai65}, and for the continuing discussions I had with him on the topic. I would also like to thank Pedro Pinto for his suggestions.

This work has been supported by a grant of the Romanian Ministry of Research, Innovation and Digitization, CNCS/CCCDI -- UEFISCDI, project number PN-III-P1-1.1-PD-2019-0396, within PNCDI III.

\end{document}